\documentclass{article}
\usepackage{tempora}
\usepackage{newtxmath}

\usepackage{amsthm}
\usepackage{mathtools}
\usepackage{microtype}
\title{A doubling construction for Williamson matrices}
\author{Curtis Bright\\University of Waterloo}
\newtheorem{theorem}{Theorem}

\DeclareMathOperator{\PAF}{PAF}
\DeclareMathOperator{\PCF}{PCF}
\let\brace\undef
\DeclarePairedDelimiter{\floor}{\lfloor}{\rfloor}
\DeclarePairedDelimiter{\brace}{\lbrace}{\rbrace}
\DeclarePairedDelimiter{\paren}{\lparen}{\rparen}

\newcommand{\shuffle}{\mathbin{\text{\fontencoding{OT2}\selectfont x}}}

\usepackage[round]{natbib}
\usepackage[hidelinks]{hyperref}

\begin{document}
\maketitle
\begin{abstract}
A construction that generates Williamson matrices of order $2n$
from Williamson matrices of odd order $n$ is presented.
The construction is completely constructive and only uses three
simple sequence operations.
\end{abstract}

\section{Introduction}

Four square, symmetric, and circulant matrices of order $n$ with $\pm1$ entries
are known as \emph{Williamson matrices} if they satisfy
\[ A^2 + B^2 + C^2 + D^2 = 4nI_n \]
where $I_n$ is the identity matrix of order $n$.  Such matrices were
first introduced by \cite{williamson1944hadamard},
who proved that such matrices can be used to constuct a
Hadamard matrix (a square matrix with $\pm1$ entries
whose rows are pairwise orthogonal) of order~$4n$.
Since Williamson matrices are circulant they are defined in terms of their
first row and so it is convenient to instead think of Williamson matrices in terms of four
sequences $(a_0,\dotsc,a_{n-1})$, $(b_0,\dotsc,b_{n-1})$, $(c_0,\dotsc,c_{n-1})$,
$(d_0,\dotsc,d_{n-1})$.
Since Williamson matrices are symmetric these sequences are also symmetric
(i.e., $x_k=x_{n-k}$ for $k=1$, $\dotsc$, $n-1$).

\section{Preliminaries}

A doubling construction for Hadamard matrices was originally given by
\cite{sylvester1867} who showed that a Hadamard matrix of order~$2n$
can be constructed from a Hadamard matrix of order~$n$.
\cite{baumert1965hadamard} provided a doubling construction for
generalizations of Williamson matrices which are often referred to as
\emph{Williamson-type} matrices~\cite[Def.~3.3]{seberry1992hadamard}.
Using complex Hadamard matrices, \cite{turyn1970} provided a construction which
generates Williamson matrices of order $2^kn$ for $k=1$, $2$, $3$, $4$
from Williamson matrices of odd order $n$.
In this paper we provide a simple doubling construction
which works directly on the sequences which define Williamson matrices.

\subsection{Correlation}

Williamson matrices can also be defined in terms of a correlation function.
The \emph{periodic cross-correlation function} of two sequences
$X=(x_0,\dotsc,x_{n-1})$ and $Y=(y_0,\dotsc,y_{n-1})$ is defined to be
\[ \PCF_{X,Y}(s) \coloneqq \sum_{k=0}^{n-1} x_k y_{k+s\bmod n} \]
and the \emph{periodic autocorrelation function} of $X$ be a sequence is
$\PAF_X(s) \coloneqq \PCF_{X,X}(s)$.
In~\cite[\S3.1.1]{brightthesis} it is shown that four symmetric sequences
$A$, $B$, $C$, $D\in\brace{\pm1}^n$ form the initial rows of
a set of Williamson matrices if and only if they satisfy
\[ \PAF_A(s) + \PAF_B(s) + \PAF_C(s) + \PAF_D(s) = 0 \]
for $s=1$, $\dotsc$, $\floor{n/2}$.
We refer to such sequences as \emph{Williamson sequences}.

\subsection{Sequence operations}\label{sec:seqop}

Let $A=(a_0,\dotsc,a_{n-1})$ and $B=(b_0,\dotsc,b_{n-1})$
be sequences of order $n$.  Our construction
uses the following 3 types of operations.
\begin{enumerate}
\item Negation.  Individually negate each entry of $A$, i.e., $-A\coloneqq(-a_0,\dotsc,-a_{n-1})$.
\item Shift.  Cyclically shift the entries of $A$ by an offset of $k$,
i.e., $(a_k,a_{k+1},\dotsc,a_{k-1})$ with indices taken modulo $n$.
\item Interleave.  Interleave the entries of $A$ and $B$ in a perfect shuffle, i.e.,
\[ A\shuffle B \coloneqq (a_0,b_0,a_1,b_1,\dotsc,a_{n-1},b_{n-1}) . \]
\end{enumerate}
If $n$ is odd we let $A'$ denote shifting $A$ by an offset of $(n-1)/2$, i.e.,
\[ A' \coloneqq (a_{(n-1)/2},\dotsc,a_{n-1},a_0,a_1,\dotsc,a_{(n-3)/2}) . \]
Note that we have $\PAF_{-A}(s)=\PAF_A(s)$, $\PAF_{A'}(s)=\PAF_A(s)$, and
\[ \PAF_{A\shuffle B}(s) = \begin{cases}
\PAF_A(s/2)+\PAF_B(s/2) & \text{when $s$ is even}, \\
\PCF_{A,B}\paren[\big]{\frac{s-1}{2}}+\PCF_{B,A}\paren[\big]{\frac{s+1}{2}} & \text{when $s$ is odd} .
\end{cases} \]

\section{Doubling construction}

Our doubling construction is captured by the following theorem.

\begin{theorem}
Let $A$, $B$, $C$, $D$ be Williamson sequences of odd order $n$.
Then
{\rm\[ A\shuffle B',\, (-A)\shuffle B',\, C\shuffle D',\, (-C)\shuffle D' \]}%
are Williamson sequences of order $2n$.
\end{theorem}
\begin{proof}
The fact that the constructed sequences have $\pm1$ entries are of length $2n$
follows directly from the properties of the three types of operations used to generate them.
The fact that they are symmetric follows from the fact that the
sequences $X$ which appear to the left of $\shuffle$ satisfy $x_k=x_{n-k}$
for $k=1$, $\dotsc$, $n-1$ and the sequences $Y$ which appear to the right
of $\shuffle$ satisfy $y_k=y_{n-k-1}$ for $k=0$, $\dotsc$, $n-1$ which are
exactly the necessary properties for $X\mspace{-1mu}\shuffle\mspace{1mu}Y$ to be symmetric.

Let $L$ be the list containing the constructed sequences of order $2n$.
To show these sequences are Williamson we need to show that
\[ \sum_{X\in L}\PAF_X(s) = 0 \]
for $s=1$, $\dotsc$, $n$.  When $s$ is even and in this range using the properties from Section~\ref{sec:seqop}
we obtain
\[ \sum_{X\in L}\PAF_X(s) = 2\sum_{X=A,B,C,D}\PAF_X(s/2) = 0 \]
since $A$, $B$, $C$, $D$ are Williamson.  When $s$ is odd we have that
\[ \PAF_{(-X)\shuffle Y}(s) = -\PAF_{X\shuffle Y}(s) \]
and using this for $(X,Y)=(A,B')$ and $(C,D')$ derives the desired property.
\end{proof}

We remark that unlike the doubling constructions given by Sylvester
and Baumert--Hall our doubling construction cannot be applied repeatedly
because it only applies when~$n$ is odd.  When $n$ is even it is not
possible to apply a shift to a symmetric sequence~$Y$ of order~$n$ to obtain a sequence
$Y'$ which satisfies $y'_k=y'_{n-k-1}$ for $k=0$, $\dotsc$, $n-1$ and this
property is necessary to make the constructed sequences symmetric.

\bibliography{williamson-doubling}

\end{document}